\documentclass{amsart}
\usepackage[dvipsnames]{xcolor}
\usepackage{amsmath}
\usepackage{amsthm}
\usepackage{amssymb}
\usepackage{tikz-cd}
\usepackage{etoolbox}
\usepackage{enumitem}
\definecolor{DarkPurple}{RGB}{102,0,153}
\usepackage[
    colorlinks=true,
    linkcolor=OliveGreen,   
    citecolor=DarkPurple, 
    urlcolor=NavyBlue
]{hyperref}

\usepackage{tikz}
\usetikzlibrary{calc}

\newtheorem{thm}{Theorem}[section]
\newtheorem{theorem}[thm]{Theorem}
\newtheorem{corollary}[thm]{Corollary}
\newtheorem{question}[thm]{Question}
\newtheorem{lemma}[thm]{Lemma}
\newtheorem{prop}[thm]{Proposition}
\newtheorem{proposition}[thm]{Proposition}

\theoremstyle{definition}
\newtheorem{definition}[thm]{Definition}

\newtheorem{remark}[thm]{Remark}

\title[$t$-structures on the thin and thick affine flag varieties]{Perverse sheaves and $t$-structures on the thin and thick affine flag varieties}

\author{Roman Bezrukavnikov}
\author{Calder Morton-Ferguson}

\begin{document}

\begin{abstract}
    We study the categories $\mathrm{Perv}_{\mathrm{thin}}$, $\mathrm{Perv}_{\mathrm{thick}}$ of Iwahori-equivariant perverse sheaves on the thin and thick affine flag varieties associated to a split reductive group $G$. 
    An earlier work of the first author 
    describes $\mathrm{Perv}_{\mathrm{thin}}$ in terms of bimodules over the so-called noncommutative Springer resolution.
    We partly extend this result to $\mathrm{Perv}_{\mathrm{thick}}$,
    providing a similar description for its anti-spherical quotient.
    The long intertwining functor realizes $\mathrm{Perv}_{\mathrm{thick}}$ as the Ringel dual of $\mathrm{Perv}_{\mathrm{thin}}$; we point out that it shares some exactness properties with the similar functor acting on perverse sheaves on the finite-dimensional flag variety.
     We use this result to resolve a conjecture of Arkhipov and the first author, proving that the image in the Iwahori-Whittaker category of any convolution-exact perverse sheaf on the affine flag variety is tilting.
\end{abstract}

\maketitle

\setcounter{tocdepth}{1}
\tableofcontents

\section{Introduction}

In the study of the BGG Category $\mathcal{O}$ via its geometric incarnation as perverse sheaves on the flag variety $G/B$, an important tool is the \emph{Radon transform}, or \emph{long intertwining functor}, studied, for example, in \cite{BeilGin}, \cite{Soerg}, \cite{BBM}. This is an endofunctor of $D^b(\mathcal{O})$ satisfying many favorable properties: it maps injective perverse sheaves to tilting perverse sheaves, tiltings to projectives, and its square is the Serre functor for $D^b(\mathcal{O})$. 
In particular, it exhibits \emph{Ringel self-duality}
of the category $\mathcal O$, i.e. provides an equivalence between that
category and its Ringel dual (see 
 e.g.\ \cite{Koenig},\ \cite{CouMaz},\ \cite{BraMau} for discussion of Ringel duality).

An analogue of this functor exists for the affine flag variety,
see e.g.\ \cite{Soerg} and \cite{Y}. In this case, the Radon transform is no longer an endofunctor of the derived category, but a correspondence between Iwahori-equivariant constructible sheaves on the affine flag variety (which we will always refer to as $X$) and those on the \emph{thick flag variety} (which we refer to as $Y$), an infinite-dimensional variety first considered by Kashiwara in\ \cite{Kash}.

The category $D^b_I(X)$ of Iwahori-equivariant constructible sheaves on the affine flag variety is a categorification of the affine Hecke algebra and was studied in  \cite{AB} and\ \cite{B} along with the Iwahori-Whittaker category $D_{\mathrm{IW}}^b(X)$,
which is a categorification of the anti-spherical module.
In particular, it is shown there that $D_{\mathrm{IW}}^b(X)$
is derived equivalent to the category of $G^\vee$-equivariant modules over a certain noncommutative algebra $A$ with an action 
of the dual group $G^\vee$, while $D^b_I(X)$ is equivalent
to a category of $G^\vee$-equivariant bimodules over $A$
(the algebra $A$ is referred to as the noncommutative Springer resolution).
The subcategories of perverse sheaves in $D^b_I(X)$, $D_{\mathrm{IW}}^b(X)$ also admit a description in these terms:
they turn out to be equivalent to the category of {\em perverse coherent equivariant} (bi)modules over $A$ with respect to the middle perversity.

One can use the above equivalences to equip $D_{\mathrm{IW}}^b(X)$
with another $t$-structure, the image of the tautological $t$-structure on equivariant $A$-modules under the equivalence. We will refer to it as the noncommutative Springer (NCS) $t$-structure. Similarly, we use the same terminology to refer to the image in $D_I^b(X)$ of the tautological $t$-structure on equivariant $A$-bimodules.


The first main result of this paper relates these two ideas: the affine Radon transform and the noncommutative Springer resolution. It shows that the Radon transform on the Iwahori-Whittaker category exchanges this NCS $t$-structure with the perverse $t$-structure, thereby relating the two via Ringel duality.
\begin{theorem}\label{thm:tstructure}
    There is a well-defined $t$-structure on $D^b_{\mathrm{IW}}(X)$ which maps under the Radon transform $R_{X \to Y}$ to the perverse $t$-structure on $D^b_{\mathrm{IW}}(Y)$. It is equal to the \emph{noncommutative Springer $t$-structure} $({}^{\mathrm{mod}}D^{\leq 0}, {}^{\mathrm{mod}}D^{\geq 0})$ arising from the tautological $t$-structure on the category of modules over the noncommutative Springer resolution algebra $A$.
\end{theorem}
The analogue of Theorem \ref{thm:tstructure} for the Iwahori-equivariant category $D_{I}^b(X)$ (as opposed to the Iwahori-Whittaker category $D_{\mathrm{IW}}^b(X)$) is false. Though the Radon transform for the finite-dimensional flag variety is a derived equivalence, in the case of the affine flag variety it is not essentially surjective as a functor between bounded derived categories, so the preimage of the perverse $t$-structure is not well-defined on $D^b_I(X)$. Accordingly, the bulk of the proof of Theorem \ref{thm:tstructure} is an argument that, in the special case of Iwahori-Whittaker sheaves, one can lift the perverse $t$-structure on thick flags under the Radon transform functor $R_{X \to Y}^{\mathrm{IW}}$ to $D_{\mathrm{IW}}^b(X)$ in a well-defined way.

This serves as one of two examples in this paper of a result which holds for the finite-dimensional flag variety, but whose generalization to the affine flag variety only holds in the Iwahori-Whittaker context. The second is the following, which will serve as our main focus for the second half of this paper. It is a result of Mirkovi{\'c} (c.f. \cite[Remark 7]{AB}) that any convolution-exact perverse sheaf on the usual flag variety (smooth along the Schubert stratification) is tilting. To recall this notion, a \emph{tilting object} is an object admitting both a standard and a costandard filtration. However, Gaitsgory's central sheaves introduced in \cite{GaitsgoryCentral} serve as a class of examples of convolution-exact perverse sheaves on the affine flag variety which are not always tilting, showing that the na{\"i}ve generalization of this fact to the affine flag variety is false. That being said, in\ \cite[Theorem 7]{AB}, the authors show that the image of central sheaves under the Iwahori-Whittaker averaging functor $\mathrm{Av}_{\psi} : D_I^b(X) \to D_{\mathrm{IW}}^b(X)$ are tilting. Accordingly, the authors of loc.\ cit.\ asked the following question.
\begin{question}[\cite{AB}, Remark 10]\label{question}
    For any convolution-exact perverse sheaf $\mathcal{F} \in \mathrm{Perv}_I(X)$, must its image $\mathrm{Av}_{\psi}(\mathcal{F})$ in the Iwahori-Whittaker category $\mathrm{Perv}_{\mathrm{IW}}(X)$ be tilting?
\end{question}

This question was studied from a combinatorial point of view in \cite{Peng}, but no definitive answer was obtained. In the second half of this paper, we use our Theorem \ref{thm:tstructure} to give an affirmative answer to Question \ref{question}.

\begin{theorem}\label{thm:convexacttilt}
    For any convolution-exact perverse sheaf $\mathcal{F} \in \mathrm{Perv}_I(X)$, the object $\mathrm{Av}_{\psi}(\mathcal{F}) \in \mathrm{Perv}_{\mathrm{IW}}(X)$ is tilting.
\end{theorem}

Any tilting perverse sheaf in $\mathrm{Perv}_{\mathrm{IW}}(X)$ is always convolution-exact, and so Theorem \ref{thm:convexacttilt} shows that convolution-exact perverse sheaves in $\mathrm{Perv}_{\mathrm{IW}}(X)$ are exactly those which are tilting. This raises the question of whether all convolution-exact perverse sheaves in $\mathrm{Perv}_I(X)$ can be detected after their projection to $\mathrm{Perv}_{\mathrm{IW}}(X)$. In Section \ref{sec:conservative}, we prove the following result.

\begin{theorem}
    The functor $\mathrm{Av}_{\psi}$ is conservative when restricted to the subcategory of $\mathrm{Perv}_I(X)$ consisting of convolution-exact objects.
\end{theorem}

This together with Theorem \ref{thm:convexacttilt} shows that any convolution-exact perverse sheaf on $\mathrm{Perv}_I(X)$ projects to a \emph{nonzero} tilting perverse sheaf in $\mathrm{Perv}_{\mathrm{IW}}(X)$ under the functor $\mathrm{Av}_{\psi}$.

Finally, in Section \ref{sec:bfo}, we recall a result of Bezrukavnikov-Finkelberg-Ostrik from \cite{BFO} which describes the action of the long intertwining functor on the finite Hecke category in terms of two-sided cells for the Weyl group. We explain how our Theorem \ref{thm:tstructure} can be viewed as an analogue of this result in the setting of Iwahori-Whittaker sheaves on the affine flag variety, and discuss a partial generalization to Iwahori-monodromic sheaves on the affine flag variety. To do so, we use a comparison appearing in \cite[Section 11]{B} of the NCS and perverse $t$-structures on $D_{I}^b(X)$ in terms of two-sided cells and nilpotent orbits.

\subsection*{Acknowledgments}
We thank Gurbir Dhillon, Ivan Losev, Alan Peng, and Xinwen Zhu for helpful discussions. The work of R. Bezrukavnikov was partly supported by NSF grant DMS-2101507. The work of C.\ Morton-Ferguson was partly supported by an NSERC PGS-D award.

\section{Preliminaries}

Let $G$ be a reductive group over a finite field $k$ or its algebraic closure with Lie algebra $\mathfrak{g}$ and Langlands dual group $G^\vee$ with Lie algebra $\mathfrak{g}^\vee$. We fix a Borel subgroup $B$ of $G$ and an opposite Borel subgroup $B^-$. Let $W_f$ be the Weyl group and $W$ the affine Weyl group, with ${}^fW$ the set of minimal-length representatives of right cosets of $W_f$ in $W$. Let $S \subset W$ be the set of simple reflections, and let $w_0 \in W_f$ be the longest element of the finite Weyl group.

Let $F = k((t))$, $\mathcal{O}^+ = k[[t]] \subset F$, and $\mathcal{O}^- = k[z^{-1}] \subset F$. Let $I \subset G(F)$ be the Iwahori subgroup given by the preimage of $B$ under the projection $G(\mathcal{O}^+) \to G$, and let $I_u$ be its pro-unipotent radical. Let $X = G(F)/I$ be the affine flag variety (which we sometimes call the \emph{thin} affine flag variety) and $\widetilde{X} = G(F)/I_u$ the extended affine flag variety, both of which are ind-schemes. We let $D_I^b(X)$ be the $I$-equivariant derived category of $\ell$-adic sheaves (for $\ell \neq \mathrm{char}(k)$) on $X$ (as explained in \cite[Section 1.2.1]{B}). Let $({}^pD^{\leq 0}, {}^pD^{\geq 0})$ be the perverse $t$-structure with heart ${}^pD^{\heartsuit} = \mathrm{Perv}_I(X)$.

Now let $I^- \subset G(\mathcal{O}^-)$ be the preimage of $B^-$ under the projection $G(\mathcal{O}^-) \to G$. We let $Y = G(F)/I^-$ be the \emph{thick affine flag variety}, which we sometimes simply call the \emph{thick flag variety}. This is a scheme of infinite type considered in \cite{Kash} which we explain in more detail in Section \ref{sec:thick}. The generalization of the derived category of $\ell$-adic sheaves and the perverse $t$-structure to finite type ind-schemes is straightforward, and treated in\ \cite[A.2]{GaitsgoryCentral}, hence the definition of $D_I^b(X)$ and $\mathrm{Perv}_I(X)$ provided above. Since $Y$ is of infinite type, we will explain a way to define $D_I^b(Y)$ and $\mathrm{Perv}_I(Y)$ via certain truncations of $Y$ and a gluing procedure for dg categories in Section \ref{sec:laxlimit}. Accordingly, throughout the paper we will use the fact that for any space $Z$ with an action of $I$, $D_I^b(Z)$ is a dg category, and that the six-functor formalism gives rise to morphisms of dg categories.

\subsection{Iwahori orbits on the thin and thick flag varieties}\label{sec:iwahoriorbits}

Recall that the $I$-orbits on both $X$ and on $Y$ are parameterized by the affine Weyl group $W$ (c.f.\ \cite[5.2]{Y}). For any $w \in W$, we let $X_{w} = IwI/I$ and $Y_{w} = IwI^-/I^-$. If $y \leq w$, then $X_y \subset \overline{X}_w$ and $Y_{w} \subset \overline{Y}_y$. We note that $\mathrm{dim}(X_w) = \ell(w) = \mathrm{codim}(Y_w)$.

For any $w \in W$, we let $j_{w}^X : X_w \to X$ be the natural inclusion. We then define $\Delta_w \in \mathrm{Perv}_I(X)$ by $\Delta_w = j_{w!}^X(\underline{\mathbb{Q}_{\ell}})[\ell(w)]$. Similarly, $\nabla_w = j_{w*}^X(\underline{\mathbb{Q}_{\ell}})[\ell(w)]$. Finally, $\mathrm{IC}_{w} = j_{w!*}^X(\underline{\mathbb{Q}_{\ell}})[\ell(w)]$, the Goresky-MacPherson extension of the constant sheaf on the orbit $X_w$. We note that $D_I^b(X)$ is equipped with a Verdier duality functor $\mathbb{D}$ and that $\mathbb{D}\Delta_w = \nabla_w$ for any $w \in W$ with $\mathbb{D}\mathrm{IC}_w = \mathrm{IC}_w$.

After the definition of $D_I^b(Y)$ is completed in Section \ref{sec:laxlimit}, it will be clear how to define an analogous standard object in $\mathrm{Perv}_I(Y)$ corresponding to any $w \in W$ which is supported only on the stratum $Y_w$, which we denote by $\tilde{\Delta}_w$ to distinguish it as a sheaf on $Y$ (rather than the $\Delta_w$ which live on $X$).

\subsection{Iwahori-Whittaker averaging}

Here, we follow the setup of \cite[Section 1.6]{AB}. We let $I' \subset G(\mathcal{O}^+)$ be the Iwahori subgroup obtained by taking the preimage of $B^-$ under the projection $G(\mathcal{O}^+) \to G$. This is referred to in loc.\ cit.\ as $I^-$, but we avoid this notation in the present paper due to confusion with the $I^- \subset G(\mathcal{O}^-)$ defined earlier. Let $I_u'$ be its pro-unipotent radical.

In \cite[Section 1.6]{AB}, it is shown how to define a generic character $\psi : I_u' \to \mathbb{G}_a$. We then let $D_{\mathrm{IW}}^b(X)$ be the $(I_u', \psi)$-equivariant derived category of $\ell$-adic sheaves on $X$, with $\mathrm{Perv}_{\mathrm{IW}}(X)$ its subcategory of perverse sheaves. It is shown in \cite[Lemma 1]{AB} that $D_{\mathrm{IW}}^b(X) = D^b(\mathrm{Perv}_{\mathrm{IW}}(X))$.

For any $w \in W$, let $X^w$ be the orbit of $wI$ under $I_u'$, and let $i_w : X^w \to X$ be the inclusion. For $w \in {}^fW$, there exist unique maps $\psi_w : X^w \to \mathbb{G}_a$ defined by $\psi_w(g \cdot wI) = \psi(g)$ for $g \in I_u'$. We now define for any $w \in {}^fW$ the perverse sheaves
\begin{align}
    {}^\psi \Delta_w & = i_{w!}\psi_w^*(\mathbb{AS})[\ell(w)],\\
    {}^\psi \nabla_w & = i_{w*}\psi_w^*(\mathbb{AS})[\ell(w)].
\end{align}
where $\mathbb{AS}$ is the Artin-Schreier sheaf. In the following definition and throughout the paper we use the convolution operations used in \cite{AB} and \cite{B}, c.f. Section \ref{sec:convolution}.

\begin{definition}
    Let $\mathrm{Av}_{\psi} : D_{I}^b(X) \to D_{\mathrm{IW}}^b(X)$ be defined for $\mathcal{F} \in D_{I}^b(X)$ by
    \begin{align}
        \mathrm{Av}_{\psi}(\mathcal{F}) = {}^{\psi}\Delta_e * \mathcal{F}
    \end{align}
\end{definition}

\begin{proposition}[\cite{AB}]
    The functor $\mathrm{Av}_{\psi}$ is $t$-exact and satisfies the following properties:
    \begin{enumerate}[label=(\roman*)]
        \item For any $w \in W$,\begin{align*}
           \mathrm{Av}_{\psi}(\Delta_w) = {}^\psi\Delta_{w'},\\
           \mathrm{Av}_{\psi}(\nabla_w) = {}^\psi\nabla_{w'},
        \end{align*}
        where $w'$ is the unique representative in ${}^fW$ for the left $W_{f}$-coset of $w$ in $W$.
        \item For any $w \not\in {}^fW$,
        \begin{align*}
            \mathrm{Av}_{\psi}(\mathrm{IC}_w) & = 0,
        \end{align*}
        while if $w \in {}^fW$, $\mathrm{Av}_{\psi}(\mathrm{IC}_w)$ is an irreducible object of $\mathrm{Perv}_{\mathrm{IW}}(X)$, which we call ${}^\psi\mathrm{IC}_{w}$. 
    \end{enumerate}
\end{proposition}

\begin{definition}[\cite{AB}]
    Let ${}^f\mathrm{Perv}_I(X)$ be the Serre quotient of the abelian category $\mathrm{Perv}_I(X)$ by the Serre subcategory generated under extensions by $\mathrm{IC}_w$ for $w \not\in {}^fW$.
\end{definition}

\begin{proposition}[\cite{AB}, Theorem 2]
    The functor $\mathrm{Av}_{\psi}$ induces an equivalence ${}^f\mathrm{Perv}_{I}(X) \to \mathrm{Perv}_{\mathrm{IW}}(X)$. 
\end{proposition}

Once the construction of $D_I^b(Y)$ is complete in Section \ref{sec:laxlimit}, we will see in Section \ref{sec:iwahorithick} that similar results hold for the thick flag variety, and that there is a well-defined functor $\mathrm{Av}_{\psi} : D_{I}^b(Y) \to D_{\mathrm{IW}}^b(Y)$.

\subsection{Convolution}\label{sec:convolution}

We recall that $D_{I}^b(X)$ is a monoidal category under convolution of sheaves, and that $D_{\mathrm{IW}}^b(X)$ is a right module for this category, again under convolution. One can also check that $D_I^b(Y)$ is a left module over $D_I^b(X)$ under convolution. We denote all these convolution operations by $*$. 

\begin{definition}
    We say that an object $\mathcal{F}$ in either $\mathrm{Perv}_I(X)$ or $\mathrm{Perv}_{\mathrm{IW}}(X)$ is \emph{convolution-exact} if for any $\mathcal{G} \in \mathrm{Perv}_{I}(X)$, we have that $\mathcal{F} * \mathcal{G}$ is perverse.
\end{definition}

\subsection{Two categorifications of the affine Hecke algebra}\label{sec:twocats}

We now review the results of \cite{BM} and \cite{B}, which establish an equivalence between constructible and coherent categorifications of the affine Hecke algebra.

Let $\mathcal{N}$ be the nilpotent cone in $\mathfrak{g}^\vee$ and $\pi : \tilde{\mathcal{N}} \to \mathcal{N}$ the Springer resolution, using the same notation $\pi : \tilde{\mathfrak{g}}^{\vee} \to \mathfrak{g}^\vee$ for the Grothendieck-Springer resolution. Let $\mathrm{St}' = \tilde{\mathfrak{g}}^\vee \times_{\mathfrak{g}^\vee} \tilde{\mathcal{N}}$. Until now, we have considered the $I$-equivariant derived category on the affine flag variety $X = G(F)/I$. We will now also consider the derived category $D_{I_u}(X)$ of $I_u$-monodromic sheaves on $X$, where $I_u$ is the pro-unipotent radical of $I$.

Finally, we let $A$ and $A^0$ be the \emph{noncommutative Springer algebras} from \cite{B}, first introduced in \cite{BM}, with $A^0$ a quotient of $A$. Their main role is explained in the next theorem. Each of them comes equipped with a $G^\vee$ action and is a $\mathcal{O}(\mathfrak{g}^\vee)$-algebra, i.e. a sheaf of noncommutative algebras on $\mathfrak{g}^\vee$. We let $\hat{A}$ denote the pullback of $A$ to the formal neighborhood of $0 \in \mathfrak{g}^\vee$, i.e. $\hat{A} = A \otimes_{\mathcal{O}(\mathfrak{g}^\vee)} \mathcal{O}(\hat{0})$.

\begin{theorem}[\cite{AB}, \cite{BM}, \cite{B}]\label{thm:equivs}
    There are derived equivalences
    \begin{align}
        D_{I}^b(X) \cong \mathrm{DGCoh}^{G^\vee}(\tilde{\mathcal{N}}\times_{\mathfrak{g}^\vee}^{\mathrm{L}} \tilde{\mathcal{N}})\cong D^b\left(A^0 \otimes_{\mathcal{O}(\mathfrak{g}^\vee)} (A^0)^{\mathrm{op}} - \mathrm{mod}_{\mathrm{fg}}^{G^\vee}\right)\label{eqn:dgcoh}\\
        D_{I_u}^b(X) \cong D^b\left(\mathrm{Coh}^{G^\vee}(\mathrm{St}')\right) \cong D^b\left(A \otimes_{\mathcal{O}(\mathfrak{g}^\vee)} (A^0)^{\mathrm{op}} - \mathrm{mod}_{\mathrm{fg}}^{G^\vee}\right)\label{eqn:dbcoh}
    \end{align}
    and
    \begin{align}
        D_{\mathrm{IW}}^b(X) \cong D^b(\mathrm{Coh}^{G^\vee}(\tilde{\mathcal{N}})) \cong D^b(A^0-\mathrm{mod}^{G^\vee}_{\mathrm{fg}}),\label{eqn:iwcoh}
    \end{align}
    and these equivalences intertwine the monoidal structure on $D_I^b(X)$ via convolution with derived tensor product of complexes of coherent sheaves and modules, along with the action of $D_{I}^b(X)$ on $D_{\mathrm{IW}}^b(X)$.
\end{theorem}

For details on (\ref{eqn:dgcoh}), we refer the reader to \cite{B}. In the present paper, to avoid dealing with the technicalities of derived algebraic geometry, we will sometimes forget $I$-equivariance of an element of $D_{I}^b(X)$ and remember only the $I_u$-monodromic structure so that we can simply work with the equivalence (\ref{eqn:dbcoh}), which along with (\ref{eqn:iwcoh}) will be enough for our purposes.

For any $\mathcal{F} \in D_{I_u}^b(X)$, we will write $\mathcal{A}_{\mathcal{F}}$ and $\mathcal{M}_{\mathcal{F}}$ to denote the corresponding elements of $D^b\left(\mathrm{Coh}^{G^\vee}(\mathrm{St}')\right)$ and $D^b\left(A \otimes_{\mathcal{O}(\mathfrak{g}^\vee)} (A^0)^{\mathrm{op}} - \mathrm{mod}_{\mathrm{fg}}^{G^\vee}\right)$ respectively under the equivalences in Theorem \ref{thm:equivs}.

We will use ${}^{p}{H}^i$ to denote the $i$th cohomology functor with respect to the perverse $t$-structure on $D_{I}(X)$ or $D_{\mathrm{IW}}(X)$, ${}^{\mathrm{mod}}H^i$ for the $i$th cohomology functor with respect to the tautological $t$-structure on $D^b(A^0-\mathrm{mod}^{G^\vee}_{\mathrm{fg}})$, and ${}^{\mathrm{bimod}}H^i$ for the tautological $t$-structure on on $D^b(A \otimes_{\mathcal{O}(\mathfrak{g}^\vee)} (A^0)^{\mathrm{op}} - \mathrm{mod}_{\mathrm{fg}}^{G^\vee})$.

\begin{prop}[c.f. Proposition 52 in \cite{B}]\label{prop:commutes}
    The diagram
        \[\begin{tikzcd}
            D^b(\mathrm{Coh}^{G^\vee}(\mathrm{St}')) \arrow[r] \arrow[d, "\mathrm{pr}_{2*}"] & D_{I_u}^b(X) \arrow[d, "\mathrm{Av}_{\psi}"]\\
            D^b(\mathrm{Coh}^{G^\vee}(\tilde{\mathcal{N}})) \arrow[r] & D_{\mathrm{IW}}^b(X)
        \end{tikzcd}\]
    commutes, where $\mathrm{pr}_{2} : \mathrm{St}' \to \tilde{\mathcal{N}}$ is the natural projection.
    \end{prop}

We now explain the nature of the algebras $A$ and $A^0$ by recalling a main tool used in \cite{B} to prove the equivalences in Theorem \ref{thm:equivs}, namely the \emph{tilting generator}. In \cite[1.4.2]{BM}, it is explained that there exists a vector bundle $\mathcal{E}$ on $\tilde{\mathfrak{g}}^\vee$ and $\mathcal{E}_0 = \mathcal{E}|_{\tilde{\mathcal{N}}}$ on $\tilde{\mathcal{N}}$ such that 
\begin{align}
    A = \mathrm{End}(\mathcal{E})^{\mathrm{op}},\\
    A^0 = \mathrm{End}(\mathcal{E}_0)^{\mathrm{op}}.
\end{align}
We let $\mathcal{T}$ denote the element of $D_{\mathrm{IW}}^b(X)$ obtained from $\mathcal{E}_0$ under the equivalence (\ref{eqn:iwcoh}) of Theorem \ref{thm:equivs}. It is explained in \cite{BM} that the NCS $t$-structure $({}^{\mathrm{mod}}D^{\leq 0}, {}^{\mathrm{mod}}D^{\geq 0})$ on $D_{\mathrm{IW}}(X)$ is characterized completely by (non-equivariant) morphisms out of $\mathcal{E}_0$, i.e.
\begin{align}
    \mathcal{F} \in {}^{\mathrm{mod}}D^{\leq 0} & \iff \mathrm{RHom}^{> 0}_{D^b(\mathrm{Coh}(\tilde{\mathcal{N}}))}(\mathcal{E}_0, \mathcal{F}) = 0, \\
    \mathcal{F} \in {}^{\mathrm{mod}}D^{\geq 0} & \iff \mathrm{RHom}^{< 0}_{D^b(\mathrm{Coh}(\tilde{\mathcal{N}}))}(\mathcal{E}_0, \mathcal{F}) = 0.
\end{align}

Following \cite[Section 11.2]{B} and defining $\tilde{\mathcal{E}} = \mathrm{pr}_1^*\mathcal{E}\otimes \mathrm{pr}_2^*\mathcal{E}_0$, for $\mathrm{pr}_1 : \mathrm{St}' \to \tilde{\mathfrak{g}}$ and $\mathrm{pr}_2 : \mathrm{St}' \to \tilde{\mathcal{N}}$ the natural projections, we note that $A \otimes_{\mathcal{O}(\mathfrak{g}^\vee)} (A^0)^{\mathrm{op}} \cong \mathrm{End}(\tilde{\mathcal{E}})$, and we have similar equivalences
\begin{align}
    \mathcal{F} \in {}^{\mathrm{bimod}}D^{\leq 0} & \iff \mathrm{RHom}^{> 0}_{D^b(\mathrm{Coh}(\mathrm{St}'))}(\tilde{\mathcal{E}}, \mathcal{F}) = 0,\label{eqn:bimodleq} \\
    \mathcal{F} \in {}^{\mathrm{bimod}}D^{\geq 0} & \iff \mathrm{RHom}^{< 0}_{D^b(\mathrm{Coh}(\mathrm{St}'))}(\tilde{\mathcal{E}}, \mathcal{F}) = 0.\label{eqn:bimodgeq}
\end{align}
We let $\tilde{\mathcal{T}} \in D_{I_u}^b(X)$ be the element obtained from $\tilde{\mathcal{E}}$ under the equivalence (\ref{eqn:dbcoh}) of Theorem \ref{thm:equivs}.

\begin{proposition}[\cite{BM}]\label{prop:tiltinggen}
    The object $\mathcal{T} \in D_{\mathrm{IW}}^b(X)$ is a tilting perverse sheaf.
\end{proposition}
\begin{proof}
    In \cite[2.4]{BM}, the object $\mathcal{E}$ is constructed by repeatedly applying ``reflection functors" $\mathcal{R}_{s}$ corresponding to simple reflections $s \in S$ to the coherent sheaf $\mathcal{O}_{\tilde{\mathfrak{g}}^\vee}$. It follows from 2.3.1 of loc.\ cit.\ that under the equivalences in Theorem \ref{thm:equivs}, these reflection functors are mapped to convolution functors $- * \mathcal{T}_s$ with the usual tilting perverse sheaves $\mathcal{T}_s$ which for any $s \in S$ is defined as the unique nontrivial extension
    \begin{align}
        0 \to \Delta_s \to \mathcal{T}_s \to \Delta_e \to 0.
    \end{align}
Further, $\mathcal{O}_{\tilde{\mathfrak{g}}^\vee}$ is mapped to $\Delta_e$. Since this expresses $\mathcal{T}$ as some iterated convolution of tilting perverse sheaves, and a convolution of tilting sheaves is always perverse and tilting, we must also have that $\mathcal{T}$ is perverse and tilting in $\mathrm{Perv}_{\mathrm{IW}}(X)$.
\end{proof}

\begin{proposition}\label{prop:tiltingprojectivemod}
    Any tilting object $T \in \mathrm{Perv}_{\mathrm{IW}}(X)$ maps to a projective object in ${}^{\mathrm{mod}}D^{\heartsuit}$ under the equivalence (\ref{eqn:iwcoh}).
\end{proposition}
\begin{proof}
    Any tilting object $T \in \mathrm{Perv}_{\mathrm{IW}}(X)$ is a direct sum of direct summands of a convolution of the form
    \begin{align}\label{eqn:tiltingconv}
        {}^{\psi}\Delta_{e} * \mathcal{T}_{s_{i_1}} * \dots * \mathcal{T}_{s_{i_k}}
    \end{align}
    for the sheaves $\mathcal{T}_{s_{i_k}}$ as in the proof of the preceding proposition. On the coherent side of (\ref{eqn:iwcoh}), the expression (\ref{eqn:tiltingconv}) can be written as
    \begin{align}
        \mathcal{R}_{s_{i_1}} \dots \mathcal{R}_{s_{i_k}}(\mathcal{O}_{\tilde{\mathfrak{g}}^\vee}).
    \end{align}
    Since it is shown in \cite{AB} that $\mathcal{O}_{\tilde{\mathfrak{g}}^\vee}$ is a direct summand of $\mathcal{E}$, it must be projective when considered as an element of ${}^{\mathrm{mod}}D^{\heartsuit}$. We now recall that for any $s \in S$, the functor $\mathcal{R}_{s}$ is $t$-exact for the NCS $t$-structure and self-adjoint, which means it must map projectives to projectives. This implies that the object in (\ref{eqn:tiltingconv}) is mapped to a projective element of ${}^{\mathrm{mod}}D^{\heartsuit}$ under the equivalence (\ref{eqn:iwcoh}), and therefore the same is true for $T$.
\end{proof}

\subsection{Gaitsgory's central sheaves} In \cite{GaitsgoryCentral}, Gaitsgory uses the geometric Satake isomorphism and the nearby cycles functor to construct a central functor from $\mathrm{Rep}(G^\vee)$ to $D_{I}^b(X)$. As a result, for any $V \in \mathrm{Rep}(G^\vee)$, there exists a perverse sheaf $\mathcal{Z}_V$ which is convolution-exact and central, meaning for all $\mathcal{F} \in D_I^b(X)$, there are natural isomorphisms
\begin{align}
    \mathcal{Z}_V * \mathcal{F} \cong \mathcal{F} * \mathcal{Z}_V
\end{align}
which together satisfy certain natural compatibilities (as explained in \cite[Section 2.1]{BezTensor}). In \cite{AB} and \cite{B}, it is shown how to interpret central sheaves on the coherent side of the equivalences in \ref{thm:equivs}.
\begin{proposition}[\cite{B}]
    If $\mathcal{F} \in D_{I_u}^b(X)$ maps to $\mathcal{A}_\mathcal{F} \in D^b(\mathrm{Coh}^{G^\vee}(\mathrm{St}'))$, then for any $V \in \mathrm{Rep}(G^\vee)$, $\mathcal{F} * \mathcal{Z}_V$ maps to $\mathcal{A}_{\mathcal{F}}\otimes V$ under the same equivalence.
\end{proposition}
In other words, convolution with a central sheaf can be interpreted as taking the tensor product with a $G^\vee$-representation; we then see that this operation is $t$-exact not only with respect to the perverse $t$-structure but also with respect to the tautological $t$-structure on coherent sheaves and the NCS $t$-structure arising from bimodules over the noncommutative Springer resolution.

It is explained in \cite{AB} that central sheaves admit a \emph{Wakimoto filtration}, but $\mathcal{Z}_V$ is not tilting for $V \neq 0$. That being said, the following result was shown in \cite{AB}.

\begin{theorem}[Theorem 7 in \cite{AB}]
    For any $V \in \mathrm{Rep}(G^\vee)$,
    \begin{align}
        \mathrm{Av}_{\psi}(\mathcal{Z}_V) \in \mathrm{Perv}_{\mathrm{IW}}(X)
    \end{align}
    is tilting.
\end{theorem}
To prove this theorem, the authors show directly that the result holds for $V$ minuscule or quasiminuscule, and then show that if it holds for a pair of representations, it must also hold for their tensor product. One of the consequences of our Theorem \ref{thm:convexacttilt} is a more direct proof of this result without reference to (quasi)minuscule representations, which instead uses only the convolution-exactness property of $\mathcal{Z}_V$.

The following lemma will be a useful tool throughout the paper.
\begin{lemma}\label{lem:centralrhom}
    If $\mathcal{F}, \mathcal{G} \in D_{I_u}^b(X)$ are such that for any $V \in \mathrm{Rep}(G^\vee)$,
    \begin{align}
        \mathrm{R}^i\mathrm{Hom}_{D_{I_u}^b(X)}(\mathcal{F} * \mathcal{Z}_V, \mathcal{G}) = 0,
    \end{align}
    then
    \begin{align}
        \mathrm{R}^i\mathrm{Hom}_{A\otimes_{\mathcal{O}(\mathfrak{g}^\vee)} (A^0)^{\mathrm{op}}}(\mathcal{M}_{\mathcal{F}}, \mathcal{M}_{\mathcal{G}}) = 0,
    \end{align}
    where $\mathcal{M}_{\mathcal{F}}$ and $\mathcal{M}_{\mathcal{G}}$ are the images of $\mathcal{F}$ and $\mathcal{G}$ under the equivalence (\ref{eqn:dbcoh}).
\end{lemma}

\begin{proof}
    Suppose $\mathcal{F}, \mathcal{G} \in D_{I_u}^b(X)$ satisfy the hypothesis in the lemma for all $V \in \mathrm{Rep}(G^\vee)$. For any $i \in \mathbb{Z}$, let $V_i$ be the space $\mathrm{R}^i\mathrm{Hom}(\mathcal{M}_{\mathcal{F}}, \mathcal{M}_{\mathcal{G}})$ considered as a $G^\vee$-representation. The equivalence (\ref{eqn:dbcoh}) shows that we then have
    \begin{align}
    (V_i \otimes V_i^*)^{G^\vee} & \cong (\mathrm{R}^i\mathrm{Hom}_{A\otimes_{\mathcal{O}(\mathfrak{g}^\vee)} (A^0)^{\mathrm{op}}}(\mathcal{M}_{\mathcal{F}}, \mathcal{M}_{\mathcal{G}}) \otimes V_i^*)^{G^\vee}\\
    & \cong \mathrm{R}^i\mathrm{Hom}_{A\otimes_{\mathcal{O}(\mathfrak{g}^\vee)} (A^0)^{\mathrm{op}}}(\mathcal{M}_{\mathcal{F}} \otimes V_i, \mathcal{M}_{\mathcal{G}})^{G^\vee}\\
    & \cong \mathrm{R}^i\mathrm{Hom}_{D_{I_u}^b(X)}(\mathcal{F} * \mathcal{Z}_{V_i}, \mathcal{G}) = 0.
    \end{align}
Since $(V_i \otimes V_i^*)^{G^\vee}$ is only trivial when $V_i = 0$, this means \[\mathrm{R}^i\mathrm{Hom}_{A\otimes_{\mathcal{O}(\mathfrak{g}^\vee)} (A^0)^{\mathrm{op}}}(\mathcal{M}_{\mathcal{F}}, \mathcal{M}_{\mathcal{G}}) = V_i = 0.\]
\end{proof}
 
\section{The Radon transform and the thick flag variety}

\subsection{The thick flag variety\label{sec:thick}}

In \cite[5.2]{Y}, the author explains in detail how one can make sense of the Radon transform $R_{X \to Y}$ as a correspondence between the affine flag variety $X$ and the thick flag variety $Y$.

To do so, one can truncate these spaces as follows. Choosing any $w \in W$, consider the closed projective subscheme $X_{\leq w} \subset X$ and the open subscheme $Y_{\leq w}$ of $Y$. Then by \cite{Kash}, there is a principal congruence subgroup $K_w \subset G(\mathcal{O})$ which acts trivially on $X_{\leq w}$ and freely on $Y_{\leq w}$. Then consider the quotient $K_w \backslash Y_{\leq w}$, which is then a scheme of finite type which we call $Z_{\leq w}$. Further, $K_w$ is a normal subgroup of $I$ and so $I/K_w$ acts on both $X_{\leq w}$ and $Y_{\leq w}$, and the $I/K_w$-orbits on $Z_{\leq w}$ are indexed by $y \in W$ with $y \leq w$.

We now expand on the material in \cite{Y} to explicitly define what we mean by constructible and perverse sheaves on the infinite-type scheme $Y$. 

\subsection{Sheaves and $t$-structures on the thick flag variety}\label{sec:laxlimit}

Suppose $y, w \in W$ with $y \leq w$. Then let $K_y$ and $K_w$ be principal congruence subgroups chosen as in Section \ref{sec:thick} such that $Z_{\leq y} = Y_{\leq y}/K_y$ and $Z_{\leq w} = Y_{\leq w}/K_w$. Then it is clear that $K_w \subset K_y$ is a normal subgroup, and $K_y/K_w$ is smooth.

We then let $Z_{\leq y}^w = Y_{\leq y}/K_w$, which admits a smooth map to $Z_{\leq y} = Y_{\leq y}/K_y$ by taking the further quotient of the free action of the group $K_y/K_w$. Let $\pi_{y}^w : Z_{\leq y}^w \to Z_{\leq y}$ be this map, and let $j_{y}^w : Z_{\leq y}^w \to Z_{\leq w}$ be the natural open embedding. Let $d_{y}^w$ be the relative dimension of $\pi_{y}^w$. By \cite[Proposition 3.6.1]{Achar}, since $\pi_{y}^w$ is smooth, $\pi_{y}^{w*}[d_y^w]$ is $t$-exact with respect to the perverse $t$-structure. 

We now want to define $D^b_I(Y)$ (resp. $\mathrm{Perv}_I(Y)$) to be the dg (resp. abelian) category of tuples $(\mathcal{F}_w)_{w \in W}$ where each $\mathcal{F}_w$ lies in $D^b_I(Z_{\leq w})$ (resp. $\mathrm{Perv}_{I}(Z_{\leq w})$) satisfying some compatibility. One can see the necessity of this compatibility by the intuitive notion that if $(\mathcal{F}_w)_{w \in W}$ is to represent a sheaf on the thick flag variety, one must ensure that $\pi_{y}^{w*}\mathcal{F}_y[d_y^w] = j_{y}^{w*}\mathcal{F}_w$ for any $y \leq w$. To do this, we define $D_I^b(Y)$ as the \emph{lax limit} (c.f. \cite[Section 4.1]{AG} for a discussion of why this is the choice of construction which corresponds to gluing of dg categories) of the categories $D_I^b(Z_{\leq w})$ as follows.

\begin{definition}\label{def:laxlim}
    For any $y, z \in W$ with $y \leq z$, let $\Phi_{y,z}$ be the functor
    \begin{align*}
        \Phi_{y,z} = j_{y!}^{z}\pi_{y}^{z*}[d_y^z] : D_{I}^b(Z_{\leq y}) \to D_I^b(Z_{\leq z}).
    \end{align*}
    Whenever $w \in W$ with $y \leq w \leq z$, we have a natural map
    \begin{align*}
        \Phi_{w,z} \circ \Phi_{y,w} = j_{w!}^{z}\pi_{w}^{z*}j_{y!}^{w}\pi_{y}^{w*}[d_y^z] \to j_{y!}^{z}\pi_{y}^{z*}[d_y^z] = \Phi_{y,z}
    \end{align*}
    which can be seen by proper base change. We then define $D^b_I(Y)$ as the lax limit of the dg categories $D_I^b(Z_{\leq w})$ along the functors $\Phi_{y,z}$, i.e.
    \begin{align}
        D^b_I(Y) = \text{lax-lim}_{w \in W} D^b_I(Z_{\leq w}).
    \end{align}
    By 4.1.2 of \cite{AG}, objects of $D^b_I(Y)$ can then be described as the collection of tuples $(\mathcal{F}_w)_{w \in W}$ with compatible morphisms $j_{y!}^z\pi_{y}^{z*}\mathcal{F}_y[d_y^z] \to \mathcal{F}_z$ for any $y \leq z$ (in loc.\ cit., the word ``compatible'' here is explained more rigorously).
\end{definition}

We now prove the following more general result about lax limits of dg categories with $t$-structures along right $t$-exact functors, though in this paper we will only use it to discuss the specific situation of the example in Definition \ref{def:laxlim} (along with the Iwahori-Whittaker analogue which we will consider in the next subsection). To do so, we follow the setup of \cite{AG} by letting $\mathrm{DGCat}_{\mathrm{cont}}$ be the $(\infty, 1)$-category of presentable dg categories and continuous functors. Let $I$ be an index $\infty$-category and let $i \mapsto \mathbf{C}_i$, $(\alpha : i \to j) \mapsto (\Phi_\alpha : \mathbf{C}_i \to \mathbf{C}_j)$ be a functor $I \to \mathrm{DGCat}_{\mathrm{cont}}$.
\begin{proposition}\label{prop:tstruct}
    Suppose that the dg categories $\mathbf{C}_i$ are equipped with t-structures such that the functors $\Phi_{\alpha}$ are each right $t$-exact. Then the dg category
    \begin{align}
        \mathbf{C} = \emph{lax-lim}_{i \in I} \mathbf{C}_i
    \end{align}
    carries a well-defined $t$-structure given by the pullback of the $t$-structures on each $\mathbf{C}_i$ under the maps $\mathbf{C} \to \mathbf{C}_i$.
\end{proposition}
\begin{proof}
    The only axiom of a $t$-structure which is nontrivial to check and which will rely on right $t$-exactness of the functors $\Phi_{\alpha}$ is the existence of truncation functors and their corresponding distinguished triangles. For any $A \in \mathbf{C}$, we can write it as a compatible system of $A_i \in \mathbf{C}_i$ with compatible morphisms $\Phi_{\alpha}A_i \to A_j$ for every $\alpha : i \to j$. 

    We let $(\tau_{\leq 0}^i, \tau_{\geq 0}^i)$ be the truncation functors corresponding to the $t$-structure on each $\mathbf{C}_i$. We now show that the collections $\{\tau_{\leq 0}^iA_i\}_{i \in I}$ and $\{\tau_{\geq 0}^iA_i\}_{i \in I}$ are each equipped with compatible morphisms satisfying the conditions necessary to belong to the lax limit. Indeed, the right $t$-exactness of $\Phi_\alpha$ implies that there are natural isomorphisms $\Phi_{\alpha}\tau_{\leq 0}^i \cong \tau_{\leq 0}^j \Phi_{\alpha}\tau_{\leq 0}^i$ and $\tau_{\geq 0}^j\Phi_{\alpha} \cong \tau_{\geq 0}^j\Phi_{\alpha} \tau_{\geq 0}^i$. We can use these isomorphisms to build maps
    \begin{align}
        \Phi_\alpha \tau_{\leq 0}^i A_i & \cong \tau_{\leq 0}^j \Phi_{\alpha}\tau_{\leq 0}^i A_i \to \tau_{\leq 0}^j \Phi_{\alpha}A_i \to  \tau_{\leq 0}^jA_j,\\
        \Phi_\alpha \tau_{\geq 0}^i A_i & \to \tau_{\geq 0}^j\Phi_\alpha \tau_{\geq 0}^i A_i \cong \tau_{\geq 0}^j\Phi_\alpha A_i \to \tau_{\geq 0}^j A_j,
    \end{align}
    and one can check that the compatibility of the maps $\Phi_\alpha A_i \to A_j$ implies the compatibility of the above maps, meaning that $\{\tau_{\leq 0}A_i\}_{i \in I}$ and $\{\tau_{\geq 0}A_i\}_{i \in I}$ are valid elements of $\mathbf{C}$ by its definition as a lax limit. It is a general property of lax limits of pretriangulated dg categories that the triangles in each $\mathbf{C}_i$ imply that $\tau_{\leq 0}A \to A \to \tau_{\geq 1}A$ is a distinguished triangle. 
\end{proof}

\begin{corollary}
    There is a well-defined perverse $t$-structure on $D_I^b(Y)$, whose heart we denote $\mathrm{Perv}_I(Y)$. This heart consists of tuples $(A_w)_{w \in W} \in D_I^b(Y)$ such that $A_w \in \mathrm{Perv}_I(Z_{\leq w})$ for all $w \in W$. 
\end{corollary}

For any $y, w\in W$, let $i_{\not\leq y}^w$ be the closed embedding which is complementary to the open embedding $j_y^w : Z_{\leq y}^w \to Z_{\leq w}$. 

\begin{definition}\label{def:eventually}
    We say that $\mathcal{F} = (\mathcal{F}_w)_{w\in W} \in D^b(Y)$ is \emph{eventually constant} (resp. \emph{eventually zero}) if there exists some $y \in w_0{}^fW$ for which $i_{\not\leq y}^{w*}\mathcal{F}_w$ is constant (resp. zero) on each $G(\mathcal{O})$-orbit for all $w \in W$.

    Note that if $\mathcal{F}$ is eventually zero with $y$ as above, then all of the data $(\mathcal{F}_w)_{w \in W}$ is completely determined by $\mathcal{F}_y \in Z_{\leq y}$, and so we can and will identify $\mathcal{F}$ with an element of $D^b_I(Z_{\leq y})$.
\end{definition}

\subsection{Iwahori-Whittaker sheaves on the thick flag variety}\label{sec:iwahorithick}

We now use a similar construction to obtain a well-defined notion of the Iwahori-Whittaker category $D^b_{\mathrm{IW}}(Y)$ on the thick flag variety, and an averaging functor $\mathrm{Av}_\psi : D^b_I(Y) \to D_{\mathrm{IW}}^b(Y)$. 

\begin{lemma}\label{lem:y'}
    If $y \in w_0{}^fW$, then $Y_{\leq y}$ is a union of $I'$-orbits. This means $Z_{\leq y}$ admits a well-defined action of $I'$ which factors through an action of $I'/K_y$.
\end{lemma}

This means that the category $D_{\mathrm{IW}}^b(Z_{\leq y})$ is well-defined for any $y \in w_0 {}^fW$, and that $\mathrm{Av}_{\psi}$ is a well-defined functor $D^b_I(Z_{\leq y}) \to D^b_{\mathrm{IW}}(Z_{\leq y})$ for such $y$. We note that in Definition \ref{def:laxlim}, we could just as easily have defined $D_I^b(Y)$ as the lax limit over categories indexed by only those $y \in w_0{}^fW$ rather than using all of $W$, and obtained the very same category $D_I^b(Y)$. We then define $D_{\mathrm{IW}}^b(Y)$ in this way, as the lax limit over the dg categories $D_{\mathrm{IW}}^b(Z_{\leq y})$ for $y \in w_0{}^fW$ using the same gluing functors as in Definition \ref{def:laxlim}. We then see that this alternate definition for $D_I^b(Y)$ turns the collection of Iwahori-Whittaker averaging functors $D^b_I(Z_{\leq y}) \to D^b_{\mathrm{IW}}(Z_{\leq y})$ for $y \in w_0{}^fW$ into a well-defined functor $\mathrm{Av}_{\psi} : D_I^b(Y) \to D_{\mathrm{IW}}^b(Y)$.

Just as in Proposition \ref{prop:tstruct}, there is a well-defined perverse $t$-structure on $D_{\mathrm{IW}}^b(Y)$, and we let $\mathrm{Perv}_{\mathrm{IW}}(Y)$ be its heart. We then let ${}^\psi \tilde{\Delta}_w = \mathrm{Av}_{\psi}(\tilde{\Delta}_w)$ for any $w \in W$, noting that if $w = yz$ for $y \in W_{f}$, then ${}^\psi \tilde{\Delta}_w = {}^\psi \tilde{\Delta}_z$. We can then use the same construction as in \cite{Y} for $R_{X \to Y}$ to define an Iwahori-Whittaker Radon transform functor $R_{X \to Y}^{\mathrm{IW}} : D_{\mathrm{IW}}^b(X) \to D_{\mathrm{IW}}^b(Y)$. 

\begin{definition}[\cite{Y}]
    Let $D_{\Delta, I}^b(Y)$ be the subcategory of $D_I^b(Y)$ generated under distinguished triangles by the standard objects $\tilde{\Delta}_w$ for $w \in W$.

    Similarly, let $D_{\Delta, \mathrm{IW}}^b(Y)$ be the subcategory of $D_{\mathrm{IW}}(Y)$ generated under distinguished triangles by ${}^\psi \tilde{\Delta}_w$ for $w \in {}^fW$.
\end{definition}

\begin{theorem}[\cite{Y}]\label{thm:iwequiv}
    $R_{X \to Y}$ (resp. $R_{X \to Y}^{\mathrm{IW}}$) is an equivalence of categories from $D_I^b(X)$ to $D_{\Delta, I}^b(Y)$ (resp. from $D_{\mathrm{IW}}^b(X)$ to $D_{\Delta, \mathrm{IW}}^b(Y)$) satisfying the following properties:
    \begin{enumerate}[label=(\roman*)]
        \item For any $w \in W$, \begin{align*}
            R_{X \to Y}(\nabla_{w}) & = \tilde{\Delta}_{w}\\
            R_{X \to Y}^{\mathrm{IW}}({}^\psi\nabla_{w}) & = {}^\psi\tilde{\Delta}_{w}\\
        \end{align*}
        \item The diagram
        \[\begin{tikzcd}
            D_{I}^b(X) \arrow[r, "R_{X \to Y}"] \arrow[d, "\mathrm{Av}_{\psi}"] & D_{\Delta, I}^b(Y) \arrow[d, "\mathrm{Av}_{\psi}"]\\
            D_{\mathrm{IW}}^b(X) \arrow[r, "R_{X \to Y}^{\mathrm{IW}}"] & D_{\Delta, \mathrm{IW}}^b(Y)
        \end{tikzcd}\]
        commutes.
        \item Tilting objects of $\mathrm{Perv}_I(X)$ (resp. $\mathrm{Perv}_{\mathrm{IW}}(X)$) map to projective objects of $\mathrm{Perv}_I(Y)$ (resp. $\mathrm{Perv}_{\mathrm{IW}}(Y)$) (c.f. \cite[4.2.1(1)]{Y}).
    \end{enumerate}
\end{theorem}

\subsection{A counterexample for the na{\"i}ve generalization}\label{sec:counterexamplenaive}

In $\mathrm{Perv}_I(X)$, it follows from Kazhdan-Lusztig theory that there is a surjection $\Delta_{s_1s_0} \to \mathrm{IC}_{s_1s_0}$, whose kernel $K$ fits into an exact sequence
\begin{align}
    0 \to \Delta_e \to \Delta_{s_1} \oplus \Delta_{s_0} \to K \to 0.
\end{align}
Taking Verdier duals and letting $C = \mathbb{D}K$, we get the distinguished triangles

\begin{align}
    \mathrm{IC}_{s_1s_0} \to & \nabla_{s_1s_0} \to C,\\
    C \to \nabla_{s_1} & \oplus \nabla_{s_0} \to \nabla_e,
\end{align}
in $D_I^b(X)$, which give the distinguished triangles
\begin{align}
    R_{X \to Y}(\mathrm{IC}_{s_1s_0}) \to & \tilde{\Delta}_{s_1s_0} \to R_{X \to Y}(C)\label{eqn:ic}\\
    R_{X \to Y}(C) \to \tilde{\Delta}_{s_1} & \oplus \tilde{\Delta}_{s_0} \to \tilde{\Delta}_e\label{eqn:c}
\end{align}
in $D_I^b(Y)$. The long exact sequence of perverse cohomology associated to (\ref{eqn:c}) shows that
\begin{align}
    {}^pH^1(R_{X \to Y}(C)) = \mathrm{coker}(\tilde{\Delta}_{s_1} \oplus \tilde{\Delta}_{s_0} \to \tilde{\Delta}_e),
\end{align}
while the long exact sequence of perverse cohomology associated to (\ref{eqn:ic}) shows that
\begin{align}
    {}^pH^2(R_{X\to Y}(\mathrm{IC}_{s_1s_0})) \cong {}^pH^1(R_{X\to Y}(C))
\end{align}
It is a straightforward computation that the perverse sheaf $\mathrm{coker}(\tilde{\Delta}_{s_1} \oplus \tilde{\Delta}_{s_0} \to \tilde{\Delta}_e)$ is constant, i.e. each element in its representing tuple of objects in the categories $D_I^b(Z_{\leq w})$ is constant. Since evidently the constant sheaf is not supported on finitely many strata, it cannot lie in $D_{\Delta, I}^b(Y)$, and therefore is not the image of $R_{X \to Y}$. This means ${}^p\tau_{\geq 2}R_{X \to Y}(\mathrm{IC}_{s_1s_0})$ is not in the image of $R_{X \to Y}$, so there is no natural way to lift the perverse truncation functor ${}^p\tau_{\geq 2}$ on $D^b_I(Y)$ to a well-defined truncation functor on $D_I^b(X)$.

\section{The preimage of the perverse $t$-structure under the Radon transform}

\subsection{Comparison of the $t$-structures}

\begin{proposition}\label{prop:thickinbimod}
    If $\mathcal{F} \in D_{I_u}^b(X)$ is an object satisfying $R_{X \to Y}(\mathcal{F}) \in D_{I_u}^{\leq 0}(Y)$ (resp. $R_{X \to Y}(\mathcal{F}) \in D_{I_u}^{\geq 0}(Y)$), then $\mathcal{F}$ lies in ${}^{\mathrm{bimod}}D^{\leq 0}$ (resp. ${}^{\mathrm{bimod}}D^{\geq 0}$), i.e. is coconnective (resp. connective) with respect to the noncommutative Springer $t$-structure on $D_{I_u}^b(X)$.
\end{proposition}
\begin{proof}
    Let $\mathcal{F}$ be as in the proposition. By Lemma \ref{lem:centralrhom} and the equivalences (\ref{eqn:bimodleq}) and (\ref{eqn:bimodgeq}), it is enough to show that for any $V \in \mathrm{Rep}(G^\vee)$ and any $i > 0$ (resp. $i < 0$),
    \begin{align}
        \mathrm{R}^i\mathrm{Hom}(\tilde{\mathcal{E}} * \mathcal{Z}_V, \mathcal{F}) = 0.
    \end{align}
    By the discussion in \cite[Section 2.4]{BM} applied to the Steinberg variety rather than simply to $\tilde{\mathfrak{g}}$ itself, the tilting generator $\tilde{\mathcal{E}}$ as a coherent sheaf can be described as a direct sum of iterated applications of reflection functors $R_{\alpha_i}$ to the structure sheaf $\mathcal{O}_{\mathrm{St'}}$, applied on the $\tilde{\mathfrak{g}}$ and $\tilde{\mathcal{N}}$ side of $\mathrm{St}'$. More precisely, it is shown in \cite{B} that $\mathcal{O}_{\mathrm{St}'}$ corresponds under (\ref{eqn:dbcoh}) to the constructible sheaf $\Xi$ (the indecomposable tilting object in $\mathrm{Perv}_{I_u}(X)$ indexed by $w_0$) and it is explained in \cite[Proof of Theorem 54]{B} that applying these two kinds of reflection functors $R_{\alpha_i}$ corresponds to convolving with the tilting objects $\mathcal{T}_{s_i}$ on the left and on the right. This means we can write $\tilde{\mathcal{E}}$ as a direct sum of objects of the form
    \begin{align}
        \mathcal{T}_{s_{i_k}} * \dots \mathcal{T}_{s_{i_1}} * \Xi * \mathcal{T}_{s_{j_1}} * \dots \mathcal{T}_{s_{j_k}}
    \end{align}
    for some sequences $s_{i_r}$ and $s_{j_r}$ of simple reflections. By the centrality property of $\mathcal{Z}_{V}$, we have
    \begin{align}\label{eqn:iteratedconv}
        \mathcal{T}_{s_{i_k}} * \dots \mathcal{T}_{s_{i_1}} * \Xi * \mathcal{T}_{s_{j_1}} * \dots \mathcal{T}_{s_{j_k}} * \mathcal{Z}_{V} \cong \mathcal{T}_{s_{i_k}} * \dots \mathcal{T}_{s_{i_1}} * \Xi * \mathcal{Z}_{V} * \mathcal{T}_{s_{j_1}} * \dots \mathcal{T}_{s_{j_k}}.
    \end{align}
    It is shown in \cite{AB} that $\Xi * \mathcal{Z}_V$ is tilting for any $V$ (note that for any $A$, $\Xi * A$ = $\mathrm{Av}_{I_u}\mathrm{Av}_{\psi}A$ where $\mathrm{Av}_{I_u} : D_{\mathrm{IW}}^b(X) \to D_{I_u}^b(X)$ is the $I_u$-averaging functor defined in \cite[Section 2]{AB}, c.f. Remark 3 in loc.\ cit.)

    This means the expression in (\ref{eqn:iteratedconv}) is a convolution of tilting objects, and is therefore itself tilting. This implies that for any $V$, $\tilde{\mathcal{E}} * \mathcal{Z}_V$ is in fact tilting. This, together with the results of \cite{Y}, imply that $R_{X \to Y}(\tilde{\mathcal{E}} * \mathcal{Z}_V)$ is a projective perverse sheaf in $\mathrm{Perv}_{I_u}(Y)$, meaning that for any $i > 0$ (resp. $i < 0$),
    \begin{align}
        \mathrm{R}^i\mathrm{Hom}(\tilde{\mathcal{E}} * \mathcal{Z}_{V}, \mathcal{F}) & = \mathrm{R}^i\mathrm{Hom}(R_{X \to Y}(\tilde{\mathcal{E}} * \mathcal{Z}_{V}), R_{X \to Y}(\mathcal{F})) = 0.
    \end{align}
\end{proof}

\begin{prop}\label{prop:tfae}
    For an object $\mathcal{F} \in D_{\mathrm{IW}}^b(X)$, the following are equivalent:
    \begin{enumerate}
        \item $\mathcal{F} \in {}^{\mathrm{mod}}D^{\leq 0}$ (resp. $\mathcal{F} \in {}^{\mathrm{mod}}D^{\geq 0}$),\label{eqn:modcase}
        \item For any indecomposable tilting object $\mathcal{T}_w \in \mathrm{Perv}_{\mathrm{IW}}(X)$, $\mathrm{Hom}(\mathcal{T}_w, \mathcal{F}[i]) = 0$ for any $i > 0$ (resp. $i < 0$).\label{eqn:tiltcase},
        \item $R_{X \to Y}^{\mathrm{IW}}(\mathcal{F}) \in D_{\mathrm{IW}}^{\leq 0}(Y)$ (resp. $R_{X \to Y}^{\mathrm{IW}}(\mathcal{F}) \in D_{\mathrm{IW}}^{\geq 0}(Y)$).\label{eqn:thickcase}
    \end{enumerate}
\end{prop}

\begin{proof}[Proof of Theorem \ref{thm:tstructure}] To show that (\ref{eqn:modcase}) implies (\ref{eqn:tiltcase}), suppose $\mathcal{F} \in {}^{\mathrm{mod}}D^{\leq 0}$ (resp. $\mathcal{F} \in {}^{\mathrm{mod}}D^{\geq 0}$). Then for any indecomposable tilting object $\mathcal{T}_w \in \mathrm{Perv}_{\mathrm{IW}}$ and any $i > 0$ (resp. $i < 0$), we have that
\begin{align}\label{eqn:rhomtw}
    \mathrm{Hom}_{D^b_{\mathrm{IW}}(X)}(\mathcal{T}_w, \mathcal{F}[i]) & \cong \mathrm{Hom}_{D^b(\mathrm{Coh}^{G^\vee}(\tilde{\mathcal{N}}))}(\mathcal{A}_{\mathcal{T}_w}, \mathcal{A}_\mathcal{F}[i])\\ 
    & \cong \mathrm{Hom}_{A^0-\mathrm{mod}^{G^\vee}}(\mathcal{M}_{\mathcal{T}_w}, \mathcal{M}_\mathcal{F}[i])\\
    & = 0,
\end{align}
since by Proposition \ref{prop:tiltingprojectivemod}, $\mathcal{M}_{\mathcal{T}_w}$ is a projective $A^0$-module.

Now suppose $\mathcal{F}$ satisfies (\ref{eqn:tiltcase}), i.e. $\mathrm{Hom}(\mathcal{T}_w, \mathcal{F}[i]) = 0$ for any $i > 0$ (resp. $i < 0$). We recall that by the results of \cite{Y}, $R_{X \to Y}^{\mathrm{IW}}(\mathcal{T}_w) = \tilde{\mathcal{P}}_w$, where $\tilde{\mathcal{P}}_w$ is the unique indecomposable projective in $\mathrm{Perv}_{\mathrm{IW}}(Y)$ which admits a standard filtration by $\tilde{\Delta}_{y}$ for $y \leq w$. Since $R_{X \to Y}^{\mathrm{IW}}$ is fully faithful, the condition in (\ref{eqn:tiltcase}) implies that
\begin{align}
    \mathrm{Hom}_{D_{\mathrm{IW}}^b(Y)}(R_{X \to Y}^{\mathrm{IW}}(\mathcal{T}_w), R_{X \to Y}^{\mathrm{IW}}(\mathcal{F})[i]) & = \mathrm{Hom}_{D_{\mathrm{IW}}^b(Y)}(\tilde{\mathcal{P}}_w, R_{X \to Y}^{\mathrm{IW}}(\mathcal{F})[i])
\end{align}
vanishes for $i > 0$ (resp. $i < 0$) for all $w \in {}^fW$. Since $\{\tilde{\mathcal{P}}_w\}_{w \in {}^fW}$ are a collection of projective generators for $\mathrm{Perv}_{\mathrm{IW}}(Y)$, this is equivalent to the condition that $R_{X \to Y}^{\mathrm{IW}}(\mathcal{F}) \in D_{\mathrm{IW}}^{\leq 0}(Y)$ (resp. $D_{\mathrm{IW}}^{\geq 0}(Y)$), so (\ref{eqn:thickcase}) is satisfied.

Finally, applying the same argument as in the proof of Proposition \ref{prop:thickinbimod} to the Iwahori-Whittaker setting gives that (\ref{eqn:thickcase}) implies (\ref{eqn:modcase}).

We have then shown that (\ref{eqn:modcase}) implies (\ref{eqn:tiltcase}), which implies (\ref{eqn:thickcase}), which in turn implies (\ref{eqn:modcase}), and therefore all of these conditions are equivalent.
\end{proof}

The equivalence of (\ref{eqn:modcase}) and (\ref{eqn:thickcase}) in Proposition \ref{prop:tfae} says that the NCS $t$-structure, a bounded $t$-structure on $D_{\mathrm{IW}}^b(X)$, agrees with the preimage of the perverse $t$-structure on $D_{\mathrm{IW}}^b(Y)$ under the Radon transform functor $R_{X\to Y}^{\mathrm{IW}}$; this is exactly the statement of Theorem \ref{thm:tstructure}.

\begin{remark}
    We note that in Proposition \ref{prop:tfae}, the proof of the implication (\ref{eqn:modcase}) implies (\ref{eqn:tiltcase}) is the only argument which does not generalize beyond the Iwahori-Whittaker setting to $D_{I_u}^b(X)$. This is because this implication relies on Proposition \ref{prop:tiltingprojectivemod}, which is only true in the Iwahori-Whittaker setting.
\end{remark}

\section{Convolution-exact and tilting sheaves}

We recall that whenever $\mathcal{F} \in D_{I_u}^b(X)$, we let $\mathcal{A}_{\mathcal{F}}$ and $\mathcal{M}_{\mathcal{F}}$ denote the respective elements of $D^b\left(\mathrm{Coh}^{G^\vee}(\mathrm{St}')\right)$ and $D^b\left(A \otimes_{\mathcal{O}(\mathfrak{g}^\vee)} (A^0)^{\mathrm{op}} - \mathrm{mod}_{\mathrm{fg}}^{G^\vee}\right)$ under the equivalence (\ref{eqn:dbcoh}) of Theorem \ref{thm:equivs}. We use the same notation for $\mathcal{F} \in D_{\mathrm{IW}}^b(X)$ and the equivalence (\ref{eqn:iwcoh}). When $M$ is a complex of modules on the right side of one of these equivalences, we let $\mathcal{F}_M$ and $\mathcal{A}_M$ be the image of $M$ on the left and middle of the equivalence.

\begin{lemma}\label{lem:zerosupp}
    Let $\mathcal{F} \in D_{I_u}(X)$. Then if $\mathcal{A}_\mathcal{F}$ is set-theoretically supported at the preimage of $0 \in \mathfrak{g}^\vee$ under $\mathrm{St}' \to \mathfrak{g}^\vee$, then the following are equivalent:
    \begin{enumerate}
        \item $\mathcal{M}_{\mathcal{F}}$ lies in the heart of the tautological $t$-structure on $D^b(A\otimes_{\mathcal{O}(\mathfrak{g}^\vee)}(A^0)^{\mathrm{op}}-\mathrm{mod}^{G^\vee}_{\mathrm{fg}})$,
        \item $\mathcal{F}[-\tfrac{\dim \mathcal{N}}{2}]$ lies in the heart of the perverse $t$-structure on $D_{I_u}(X)$. 
    \end{enumerate}
\end{lemma}
\begin{proof}
    This follows from the results of \cite{B}. Indeed, in loc.\ cit.\ it is explained that there is a filtration by thick subcategories on $D^b(A\otimes_{\mathcal{O}(\mathfrak{g}^\vee)}(A^0)^{\mathrm{op}}-\mathrm{mod}^{G^\vee}_{\mathrm{fg}})$ indexed by the partially-ordered set of nilpotent orbits $O \subset \mathcal{N}$. Then Theorem 54(a) of loc.\ cit.\ says that this support filtration is compatible with the image of the perverse $t$-structure under the equivalence in Theorem \ref{thm:equivs}, and that the induced $t$-structure on the associated graded category corresponding to a nilpotent orbit $O \subset \mathcal{N}$ coincides with the $t$-structure coming from the tautological one on $D^b(A\otimes_{\mathcal{O}(\mathfrak{g}^\vee)}(A^0)^{\mathrm{op}}-\mathrm{mod}^{G^\vee}_{\mathrm{fg}})$ shifted by $\tfrac{\dim \mathcal{N} - \dim O}{2}$.

    Setting $O = \{0\} \subset \mathcal{N}$, since this is orbit is minimal, the corresponding summand in the associated graded is equivalent to the full triangulated subcategory of $D^b(A\otimes_{\mathcal{O}(\mathfrak{g}^\vee)}(A^0)^{\mathrm{op}}-\mathrm{mod}^{G^\vee}_{\mathrm{fg}})$ consisting of objects such that each of their cohomology modules is set-theoretically supported at zero when considered as a module over the center $\mathcal{O}(\mathcal{N})$ (in other words, the corresponding element of $D^b(\mathrm{Coh}^{G^\vee}(\mathrm{St}'))$ is set-theoretically supported at the preimage of $0 \in \mathfrak{g}^\vee$ under $\mathrm{St}' \to \mathfrak{g}^\vee$). Given this observation, the result in the lemma is exactly the result in Theorem 54(a) of \cite{B} applied to this minimal orbit.
\end{proof}

We contextualize the relevance of \cite[Theorem 54]{B} to the results of the present paper and discuss more general statements related to the preceding lemma in Section \ref{sec:bfo}.

\begin{prop}\label{prop:convexactisbimod}
     Any convolution-exact perverse sheaf $\mathcal{F} \in \mathrm{Perv}_{\mathrm{IW}}(X)$ lies in ${}^{\mathrm{mod}}D^{\heartsuit}$. In fact, any convolution-exact perverse sheaf $\mathcal{F} \in \mathrm{Perv}_I(X)$ lies inside ${}^{\mathrm{bimod}}D^{\heartsuit}$.
\end{prop}

\begin{proof}
    Let $\mathcal{F} \in D_{\mathrm{IW}}^b(X)$ be a convolution-exact perverse sheaf.

    For any $n \geq 1$, let $A_n = A^0\otimes_{\mathcal{O}(\mathfrak{g}^\vee)} k[t]/(t^n)$, where here the tensor product corresponds to the base change along the map $\mathrm{Spec}(k[t]/t^n) \to \mathfrak{g}^\vee$ corresponding to the origin $0 \in \mathfrak{g}^\vee$. We can consider $A_n$ as a $G^\vee$-equivariant $A \otimes_{\mathcal{O}(\mathfrak{g}^\vee)} A^{\mathrm{op}}$-module and therefore as a $A \otimes (A^0)^{\mathrm{op}}$-module since $A^0$ is a quotient of $A$, and so we can freely define $\mathcal{F}_{A_n}$ to the corresponding element of $D_{I_u}(X)$.
    
    Since the equivalences in Theorem \ref{thm:equivs} are monoidal, $\mathcal{F} * \mathcal{F}_{A_n}$ corresponds to the (derived) tensor product $M \otimes_{A^0} A_n$. By Lemma \ref{lem:zerosupp}, $\mathcal{F}_{A_n}$ is perverse, i.e. lies in $\mathrm{Perv}_{I_u}(X)$. This means the convolution $\mathcal{F} * \mathcal{F}_{A_n}$ is also perverse, and has corresponding element of $D^b(\mathrm{Coh}^{G^\vee}(\tilde{\mathcal{N}}))$ supported at $0 \in \mathfrak{g}^\vee$ since $\mathcal{F}_{A_n}$ already has this property. Applying Lemma \ref{lem:zerosupp} again, this means $\mathcal{M}_{\mathcal{F}} \otimes_{A^0} A_n$ has ${}^{\mathrm{mod}}H^{\bullet}$ concentrated in degree zero.

    This gives that $M_\mathcal{F} \otimes_{A^0} A_n$ itself must have cohomology ${}^{\mathrm{mod}}H^\bullet$ concentrated in degree zero for every $n$, and therefore so must $M_\mathcal{F} \otimes_{A^0} \hat{A}$.

    To deduce the result for $\mathcal{F}$ itself, we now follow the same proof technique as was used in Lemma 2.5.2 of \cite{BM}. First note that by Proposition \ref{prop:commutes}, forgetting the bimodule structure on $M$ and only remembering its structure as a right $A^0$-module corresponds, under the derived equivalence in Theorem \ref{thm:equivs}, to pushing forward the corresponding coherent sheaf $\mathcal{A}_M$ from $\mathrm{St}'$ to $\tilde{\mathcal{N}}$ along the second projection $\mathrm{pr}_2$ and working with this pushforward $\mathrm{pr}_{2*}\mathcal{A}_M \in D^b(\mathrm{Coh}^{G^\vee}(\tilde{\mathcal{N}}))$, as we will do for the remainder of this proof.
    
    Now let $U$ be the maximal open subset of $\mathfrak{g}^\vee$ such that ${}^{\mathrm{mod}}H^i(\mathrm{pr}_{2*}M)|_{\pi^{-1}(U)} = 0$ for $i \neq 0$. Then $U$ is $G^\vee$-invariant (since $M$ is $G^\vee$-equivariant) and $U$ contains all closed points of the fiber over $0 \in \mathfrak{g}^\vee$ by the preceding argument. This means $U = \mathfrak{g}^\vee$, and so the $A^0$-module corresponding to $\mathrm{pr}_{2*}M$ under the equivalence in Theorem \ref{thm:equivs} is concentrated in degree zero.

    A similar argument shows the same result for $\mathcal{F} \in D_I^b(X)$ using the other equivalences of \cite{B}, but in the present work we will only use the first statement, i.e.\ the case of $\mathcal{F} \in \mathrm{Perv}_{\mathrm{IW}}(X)$.
\end{proof}

For a very similar result to Proposition \ref{prop:convexactisbimod} using a different method of proof in the free monodromic setting, see \cite[Lemma 5.7]{BL}. Combining Theorem \ref{thm:tstructure} with Proposition \ref{prop:convexactisbimod}, we get the following corollary.
\begin{corollary}
    For any convolution-exact perverse sheaf $\mathcal{F} \in \mathrm{Perv}_{\mathrm{IW}}(X)$, the object $R_{X \to Y}^{\mathrm{IW}}(\mathcal{F})$ lies in $\mathrm{Perv}_{\mathrm{IW}}(Y)$; in other words, $\mathcal{F} \in {}^{\mathrm{th}}D^{\heartsuit}$.
\end{corollary}

\subsection{Proof of Theorem \ref{thm:convexacttilt}}\label{sec:proofofthm}

\begin{proof}[Proof of Theorem \ref{thm:convexacttilt}]
    It is enough to prove the result for any convolution-exact $\mathcal{F} \in \mathrm{Perv}_{\mathrm{IW}}(X)$. Then $\mathcal{F} \in {}^{\mathrm{mod}}D^{\heartsuit}$ by Proposition \ref{prop:convexactisbimod}. Since ${}^{\mathrm{mod}}D^{\heartsuit}$ is equal to ${}^{\mathrm{th}}D^{\heartsuit} \subset D_{\mathrm{IW}}^b(X)$ by Theorem \ref{thm:tstructure}, we have that $R_{X \to Y}^{\mathrm{IW}}(\mathcal{F}) \in \mathrm{Perv}_{\mathrm{IW}}(Y)$.

    We can now follow the argument of Mirkovi{\'c} which appears in \cite[Remark 7]{AB} to complete the proof. Namely, we note that to show $\mathcal{F}$ is tilting, it is enough to show that $\mathcal{F}$ and $\mathbb{D}\mathcal{F}$ both lie in the triangulated subcategory of $D_{\mathrm{IW}}^b(X)$ generated by objects of the form ${}^\psi\nabla_{w}[d]$ for $w \in W$, $d \geq 0$. Since $\mathbb{D}\mathcal{F}$ is clearly also convolution-exact, without loss of generality it is enough to check this just for $\mathcal{F}$.

    Since $R_{X \to Y}^{\mathrm{IW}}\mathcal{F}$ is perverse and lies in the image $D_{\Delta, \mathrm{IW}}(Y)$ of $R_{X \to Y}^{\mathrm{IW}}$, it is an iterated extension (by distinguished triangles) of ${}^\psi\Delta_{w}[d]$ for $w \in W$, $d \geq 0$. Applying $(R_{X \to Y}^{\mathrm{IW}})^{-1}$, we obtain that $\mathcal{F}$ itself is in the subcategory of $D_{\mathrm{IW}}^b(X)$ generated by objects of the form $(R_{X\to Y}^{\mathrm{IW}})^{-1}({}^\psi\Delta_{w})[d] = {}^\psi \nabla_w$ for $w \in W$, $d \geq 0$. This means $\mathcal{F}$ has a costandard filtration. Since we can apply the same argument to $\mathbb{D}\mathcal{F}$, $\mathcal{F}$ must also have a standard filtration, and is therefore tilting.
\end{proof}

\subsection{$\mathrm{Av}_{\psi}$ is conservative on convolution-exact sheaves}\label{sec:conservative}

\begin{lemma}\label{lem:intwodegrees}
    If $\mathcal{F} \in \mathrm{Perv}_I(X)$, then for any $s \in S$, $\nabla_s * \mathcal{F} \in D^{\leq 0} \cap D^{\geq -1}$ and $\Delta_s * \mathcal{F} \in D^{\geq 0} \cap D^{\leq 1}$.
\end{lemma}

\begin{proof}
    For any $s \in S$, we have distinguished triangles
    \begin{align}
        \Delta_s & \to \mathcal{T}_s \to \Delta_e,\\
        \Delta_e & \to \mathcal{T}_s \to \nabla_s,
    \end{align}
    for $\mathcal{T}_s \in \mathrm{Perv}_I(X)$ the tilting perverse sheaf indexed by $s$. Since $\Delta_e$ and $\mathcal{T}_s$ are both convolution-exact, the result in question follows from the long exact sequence in perverse cohomology arising from these two distinguished triangles.
\end{proof}

\begin{prop}\label{prop:nobadmorphisms}
    If $\mathcal{F} \in \mathrm{Perv}_I(X)$ is convolution-exact, then for any $w \neq e$,
    \begin{align}
        \mathrm{Hom}_{\mathrm{Perv}_I(X)}(\mathrm{IC}_w, \mathcal{F}) & = \mathrm{Hom}_{\mathrm{Perv}_I(X)}(\mathcal{F}, \mathrm{IC}_w) = 0.
    \end{align}
\end{prop}

\begin{proof}
    Let $w \neq e$. We show $\mathrm{Hom}_{\mathrm{Perv}_I(X)}(\mathrm{IC}_w, \mathcal{F}) = 0$, and it then follows similarly that $\mathrm{Hom}_{\mathrm{Perv}_I(X)}(\mathcal{F}, \mathrm{IC}_w) = 0$.

    Suppose for contradiction that there exists a nonzero morphism $\mathrm{IC}_w \to \mathcal{F}$, which then must be injective. Let $\mathcal{G}$ be the cokernel of this map so that we have a distinguished triangle
        \[\begin{tikzcd}
            \mathrm{IC}_w \arrow[r] & \mathcal{F} \arrow[r] & \mathcal{G}.
        \end{tikzcd}\]
    
    Let $s \in S$ be such that $\ell(sw) < \ell(w)$. We note that $\nabla_s * \mathrm{IC}_w = \mathrm{IC}_w[1]$ and $\nabla_s * \mathcal{F}$ is perverse by the convolution-exactness of $\mathcal{F}$. The distinguished triangle
    \[\begin{tikzcd}
            \nabla_s * \mathrm{IC}_w \arrow[r] & \nabla_s * \mathcal{F} \arrow[r] & \nabla_s * \mathcal{G}
        \end{tikzcd}\]
    then gives a long-exact sequence
    \[
\begin{tikzcd}
0 \arrow[r]             & \nabla_s * \mathcal{F} \arrow[r] & {}^pH^0(\nabla_s * \mathcal{G})                \\
\mathrm{IC}_w \arrow[r] & 0 \arrow[r]                      & {}^pH^{-1}(\nabla_s * \mathcal{G}) \arrow[llu] \\
                        & \dots \arrow[r]                  & {}^pH^{-2}(\nabla_s * \mathcal{G}) \arrow[llu]
\end{tikzcd}.\]
But by Lemma \ref{lem:intwodegrees}, ${}^pH^{-2}(\nabla_s * \mathcal{G}) = 0$, and so the exactness of the above sequence would imply $\mathrm{IC}_w = 0$, a contradiction. 
\end{proof}

\begin{theorem}
    The functor $\mathrm{Av}_{\psi}$ is conservative when restricted to the subcategory of $\mathrm{Perv}_I(X)$ consisting of convolution-exact objects.
\end{theorem}
\begin{proof}
    Suppose that $f : \mathcal{F} \to \mathcal{F}'$ is a morphism between convolution-exact perverse sheaves such that $\mathrm{Av}_{\psi}(f)$ is an isomorphism. Then if $\mathcal{G} = \mathrm{cone}(f) \in D_I^b(X)$, we have $\mathrm{Av}_\psi(\mathcal{G}) = 0$ as an element of $D_{\mathrm{IW}}^b(X)$. By the definition of $\mathcal{G}$, note that we have ${}^pH^{-1}(\mathcal{G}) = \mathrm{ker} f$ and ${}^pH^0(\mathcal{G}) = \mathrm{coker} f$

    By Theorem 2 of \cite{AB}, this means $\mathcal{G}$ is in the triangulated subcategory of $D^b_I(X)$ generated by shifts of $\mathrm{IC}_w$ for $w \not\in {}^fW$. Since for any $i$, ${}^pH^i(\mathcal{G})$ is then generated under iterated extensions, kernels, and cokernels of irreducibles $\mathrm{IC}_w$ for $w \not\in {}^fW$; since the category generated by such $\mathrm{IC}_w$ under extensions is closed under all such operations, we have that ${}^pH^i(\mathcal{G})$ itself is an iterated extension of $\mathrm{IC}_w$ for $w \not\in {}^fW$. 
    
    In particular, this means that if ${}^pH^{-1}(\mathcal{G}) = \mathrm{ker} f$ is nonzero, it has an irreducible submodule $\mathrm{IC}_w$ for some $w \not\in {}^fW$. In particular, the composition
    \[\mathrm{IC}_{w} \hookrightarrow \mathrm{ker}f \hookrightarrow \mathcal{F}\]
    gives a nonzero morphism from $\mathrm{IC}_w$ to $\mathcal{F}$, which is impossible by Proposition \ref{prop:nobadmorphisms}. Similarly, if ${}^pH^0(\mathcal{G})$ is nonzero, this would yield a nonzero morphism from $\mathcal{F}'$ to some $\mathrm{IC}_y$ with $y \not\in {}^fW$, which is again impossible by Proposition \ref{prop:nobadmorphisms}. This means $\mathrm{ker} f = \mathrm{coker} f = 0$, and so $f$ is an isomorphism.
\end{proof}

\section{The long intertwining functor and two-sided cells}
In this section, we explain how Theorem \ref{thm:tstructure} (and relatedly, Proposition \ref{prop:thickinbimod}) can be thought of as an affine version of \cite[Proposition 4.1]{BFO}, which describes the behaviour of the long intertwining functor on two-sided cell subquotient categories of the ordinary (non-affine) Hecke category.

\subsection{Intertwining functors on the finite Hecke category}\label{sec:bfo}

There are many models for the finite Hecke category (Harish-Chandra bimodules are used in loc.\ cit.) but to be consistent with the setup of the rest of the present paper, we work with the category $\mathrm{Perv}_U(G/B)$ of $U$-equivariant perverse sheaves on the flag variety (where $U$ is the unipotent radical of $B$) and its derived category $\mathcal{D} = D_U^b(G/B)$. We carry over to the finite-dimensional flag variety our notation for standard and costandard objects $\Delta_w$ and $\nabla_w$ (c.f. Section \ref{sec:iwahoriorbits}) for $w \in W_f$. These objects are $B$-equivariant and so can be considered as elements of $\mathrm{Perv}_B(G/B) \subset \mathrm{Perv}_U(G/B)$.

There is a convolution product
\begin{align}
    - * - : D_U^b(G/B) \times D_B^b(G/B) \to D_U^b(G/B),
\end{align}
and following \cite{BBM} we let $I_{w_0} : D_U^b(G/B) \to D_U^b(G/B)$ be the functor defined by
\begin{align}
    I_{w_0}(\mathcal{F}) = \mathcal{F} * \Delta_{w_0},
\end{align}
where $w_0$ is the longest element of $W_f$. This is a left $t$-exact functor with respect to the perverse $t$-structure on $D_U^b(G/B)$. 

\subsection{The finite Hecke category and two-sided cells}

Let $C$ denote the poset of two-sided cells of the finite Weyl group $W_f$. For any $\underline{c}$, let $\mathcal{D}_{\leq \underline{c}} \subset \mathcal{D} = D_{U}^b(G/B)$ (resp. $\mathcal{D}_{<\underline{c}}$) be the full triangulated subcategory of $\mathcal{D}$ generated by the simple objects $\mathrm{IC}_w$ such that $w \in \underline{c}'$ for some $\underline{c}' \leq \underline{c}$ (resp. $\underline{c}' < \underline{c}$). We then let $\mathcal{D}_{\underline{c}} = \mathcal{D}_{\leq \underline{c}}/\mathcal{D}_{< \underline{c}}$ be the Serre quotient category.

By the definition of two-sided cells, convolution with any standard or costandard sheaf preserves the subcategory $\mathcal{D}_{\leq \underline{c}}$ for any $\underline{c} \in C$, so we can consider $I_{w_0}$ as an endofunctor of ${D}_{\leq \underline{c}}$. The following theorem of Bezrukavnikov-Finkelberg-Ostrik appears (using the language of Harish-Chandra bimodules rather than perverse sheaves) in \cite{BFO}.

In \cite{LCells}, Lusztig establishes a bijection between the set $\tilde{C}$ of two-sided cells of the affine Weyl group $W$ and the set of nilpotent conjugacy classes in $\mathfrak{g}^\vee$. Every cell $\underline{c} \in C$ is contained within a cell in $\tilde{C}$, so we can let $O_{\underline{c}}$ be the nilpotent orbit associated to a two-sided cell $\underline{c}$ under this equivalence. We note that under this equivalence, the value $a_{\underline{c}}$ of the ``$a$-function'' used in the original statement of Proposition 4.1 of \cite{BFO} can be expressed in terms of the dimension of $O_{\underline{c}}$, i.e. if we let
\begin{align}\label{eqn:dc}
    d_{\underline{c}} = \frac{\dim \mathcal{N} - \dim O_{\underline{c}}}{2},
\end{align}
then we have $d_{\underline{c}} = a_{\underline{c}}$ for any $\underline{c} \in C$, c.f. the proof of Proposition 4.8 in loc.\ cit. We now restate the result just mentioned in the language of perverse sheaves.

\begin{theorem}[Proposition 4.1 in \cite{BFO}]\label{thm:bfo}
    For any $\underline{c} \in C$, the functor 
    \[I_{w_0}[d_{\underline{c}}] : \mathcal{D}_{\underline{c}} \to \mathcal{D}_{\underline{c}}\]
    is $t$-exact for the perverse $t$-structure.
\end{theorem}

\subsection{Theorem \ref{thm:tstructure} as an affine analogue}

We now let $\tilde{C}$ denote the set of two-sided cells for the affine Weyl group $W$. We define $d_{\underline{c}}$ for any $\underline{c} \in \tilde{C}$ as in (\ref{eqn:dc}). We then get a filtration on $\tilde{\mathcal{D}} = D_{I_u}^b(X)$ by categories $\tilde{\mathcal{D}}_{\leq \underline{c}}$ generated by $\mathrm{IC}_{w}$ for $w \in \underline{c}' \leq \underline{c}$. We then denote by $\tilde{\mathcal{D}}_{\underline{c}}$ the cell subquotient category $\tilde{\mathcal{D}}_{\leq \underline{c}}/\tilde{\mathcal{D}}_{< \underline{c}}$.

There is a related filtration of categories called the \emph{support filtration} on the coherent side of the equivalence (\ref{eqn:dbcoh}). Following \cite[Section 11.2]{B}, we can consider the subcategory $\mathcal{A}_{\leq O}$ of $\mathcal{A} = D^b(A \otimes_{\mathcal{O}(\mathfrak{g}^\vee)} (A^0)^{\mathrm{op}} - \mathrm{mod}_{\mathrm{fg}}^{G^\vee})$ such that each cohomology module considered as a module over the center $\mathcal{O}(\mathcal{N}) \subset A\otimes_{\mathcal{O}(\mathfrak{g}^\vee)} (A^0)^{\mathrm{op}}$ is set-theoretically supported on the closure of $O$. Similarly, we define $\mathcal{A}_{< O}$ by the same condition, but with set-theoretic support only on nilpotent orbits properly contained in the closure of $O$. For any nilpotent orbit $O$, we let $\mathcal{A}_{O} = \mathcal{A}_{\leq O}/\mathcal{A}_{< O}$ be the corresponding subquotient category for this filtration.

We now explain how the perverse $t$-structure and the NCS $t$-structure are related in terms of two-sided cell subquotient categories by recalling a result from \cite{B}.
\begin{theorem}[Theorem 54(a) and Theorem 55 in \cite{B}]\label{thm:54}
    For any two-sided cell $\underline{c} \in \tilde{C}$, the restriction of the equivalence in (\ref{eqn:dbcoh}) to $\tilde{\mathcal{D}}_{\leq \underline{c}}$ gives an equivalence of categories
    \begin{align}
        \tilde{\mathcal{D}}_{\leq \underline{c}} \cong \mathcal{A}_{\leq O_{\underline{c}}}.
    \end{align}
    The $t$-structure on $\tilde{\mathcal{D}}_{\underline{c}}$ induced from the perverse $t$-structure coincides with the $t$-structure induced from the preimage of the NCS $t$-structure on $\mathcal{A}$ under the above equivalence shifted by $d_{\underline{c}}$.
\end{theorem}

In the affine setting, $R_{X \to Y}$ is the natural analogue of the functor $I_{w_0}$, but of course it is now a functor from $D_{I_u}^b(X) \to D_{I_u}^b(Y)$ rather than an endofunctor. With the above in mind, the following is then a direct consequence of our Proposition \ref{prop:thickinbimod}, we note that in this setup, it can be seen as a partial analogue of Theorem \ref{thm:bfo} in the affine situation.
\begin{corollary}
    Let $\underline{c} \in \tilde{C}$. If $\mathcal{F} \in \tilde{\mathcal{D}}_{\leq \underline{c}}$ is such that $R_{X \to Y}(\mathcal{F})[d_{\underline{c}}] \in D_{I_u}^b(Y)$ is perverse, then the image of $\mathcal{F}$ in $\tilde{\mathcal{D}}_{\underline{c}}$ is perverse.
\end{corollary}

Because we know that ${}^{\mathrm{th}}D^{\heartsuit} \neq {}^{\mathrm{bimod}}D^{\heartsuit}$, we do not see a way to make a stronger statement of this theorem in full generality to complete the analogy with Theorem \ref{thm:bfo}. However, in the Iwahori-Whittaker case, one can do so using our Theorem \ref{thm:tstructure}.

Indeed, let $\tilde{\mathcal{D}}_{\leq \underline{c}}^{\mathrm{IW}}(X)$ and $\tilde{\mathcal{D}}_{\underline{c}}^{\mathrm{IW}}(X)$ be the subcategory and subquotient category analogous to $\tilde{\mathcal{D}}_{\leq \underline{c}}$ and $\tilde{\mathcal{D}}_{\underline{c}}$, but in the Iwahori-Whittaker setting. We note that our results show that $R_{X \to Y}^{\mathrm{IW}} : D_{\mathrm{IW}}^b(X) \to D_{\mathrm{IW}}^b(Y)$ is an equivalence of bounded derived categories, so for any $\underline{c}$, we can define the triangulated category $\tilde{\mathcal{D}}_{\leq \underline{c}}^{\mathrm{IW}}(Y)$ as the image of $\tilde{\mathcal{D}}_{\leq \underline{c}}^{\mathrm{IW}}(X)$ under the equivalence $R_{X \to Y}^{\mathrm{IW}}$. We then get a functor $R_{X \to Y}^{\mathrm{IW}} : \tilde{\mathcal{D}}_{\underline{c}}^{\mathrm{IW}}(X) \to \tilde{\mathcal{D}}_{\underline{c}}^{\mathrm{IW}}(Y)$ for any $\underline{c} \in \tilde{C}$. In this affine Iwahori-Whittaker case, we can now state a full analogy to Theorem \ref{thm:bfo}.

\begin{theorem}
    For any $\underline{c} \in \tilde{C}$, the functor
    \begin{align*}
        R_{X \to Y}^{\mathrm{IW}}[d_{\underline{c}}] : \tilde{\mathcal{D}}_{\underline{c}}^{\mathrm{IW}}(X) \to \tilde{\mathcal{D}}_{\underline{c}}^{\mathrm{IW}}(Y)
    \end{align*}
    is $t$-exact for the perverse $t$-structure on $X$ and $Y$ respectively.
\end{theorem}
\begin{proof}
    Suppose $\mathcal{F} \in \tilde{\mathcal{D}}_{\underline{c}}^{\mathrm{IW}}(X)$ is perverse. By Theorem \ref{thm:54}, $\mathcal{F}[d_{\underline{c}}] \in \tilde{\mathcal{D}}_{\underline{c}}^{\mathrm{IW}}(X)$ lies in the heart of the $t$-structure on $\tilde{\mathcal{D}}_{\underline{c}}^{\mathrm{IW}}(X)$ induced from the NCS $t$-structure on $D_{\mathrm{IW}}^b(X)$. Since Theorem \ref{thm:tstructure} implies that $R_{X \to Y}^{\mathrm{IW}}$ sends the heart of the NCS $t$-structure into the heart of the perverse $t$-structure on $D_{\mathrm{IW}}^b(Y)$, we must have that $R_{X \to Y}^{\mathrm{IW}}(\mathcal{F}[d_{\underline{c}}])$ is in the perverse $t$-structure on $\tilde{\mathcal{D}}_{\underline{c}}^{\mathrm{IW}}(Y)$. Since $R_{X \to Y}^{\mathrm{IW}}$ is an equivalence of categories, the converse must also hold, implying that $ R_{X \to Y}^{\mathrm{IW}}[d_{\underline{c}}] : \tilde{\mathcal{D}}_{\underline{c}}^{\mathrm{IW}}(X) \to \tilde{\mathcal{D}}_{\underline{c}}^{\mathrm{IW}}(Y)$ is $t$-exact for the perverse $t$-structures on both sides.
\end{proof}

\clearpage

\bibliographystyle{alpha}
\bibliography{bibl}

\end{document}